\documentclass[12pt]{amsart}
\usepackage{amsmath,amssymb,amsbsy,amsfonts,amsthm,latexsym,amsopn,amstext,
                            amsxtra,euscript,amscd}

\newfont{\teneufm}{eufm10}
\newfont{\seveneufm}{eufm7}
\newfont{\fiveeufm}{eufm5}
\newfam\eufmfam
                 \textfont\eufmfam=\teneufm \scriptfont\eufmfam=\seveneufm
                 \scriptscriptfont\eufmfam=\fiveeufm
\def\bbbc{{\mathchoice {\setbox0=\hbox{$\displaystyle\rm C$}\hbox{\hbox
to0pt{\kern0.4\wd0\vrule height0.9\ht0\hss}\box0}}
{\setbox0=\hbox{$\textstyle\rm C$}\hbox{\hbox
to0pt{\kern0.4\wd0\vrule height0.9\ht0\hss}\box0}}
{\setbox0=\hbox{$\scriptstyle\rm C$}\hbox{\hbox
to0pt{\kern0.4\wd0\vrule height0.9\ht0\hss}\box0}}
{\setbox0=\hbox{$\scriptscriptstyle\rm C$}\hbox{\hbox
to0pt{\kern0.4\wd0\vrule height0.9\ht0\hss}\box0}}}}
\def\bbbq{{\mathchoice {\setbox0=\hbox{$\displaystyle\rm
Q$}\hbox{\raise
0.15\ht0\hbox to0pt{\kern0.4\wd0\vrule height0.8\ht0\hss}\box0}}
{\setbox0=\hbox{$\textstyle\rm Q$}\hbox{\raise
0.15\ht0\hbox to0pt{\kern0.4\wd0\vrule height0.8\ht0\hss}\box0}}
{\setbox0=\hbox{$\scriptstyle\rm Q$}\hbox{\raise
0.15\ht0\hbox to0pt{\kern0.4\wd0\vrule height0.7\ht0\hss}\box0}}
{\setbox0=\hbox{$\scriptscriptstyle\rm Q$}\hbox{\raise
0.15\ht0\hbox to0pt{\kern0.4\wd0\vrule height0.7\ht0\hss}\box0}}}}
\def\bbbt{{\mathchoice {\setbox0=\hbox{$\displaystyle\rm
T$}\hbox{\hbox to0pt{\kern0.3\wd0\vrule height0.9\ht0\hss}\box0}}
{\setbox0=\hbox{$\textstyle\rm T$}\hbox{\hbox
to0pt{\kern0.3\wd0\vrule height0.9\ht0\hss}\box0}}
{\setbox0=\hbox{$\scriptstyle\rm T$}\hbox{\hbox
to0pt{\kern0.3\wd0\vrule height0.9\ht0\hss}\box0}}
{\setbox0=\hbox{$\scriptscriptstyle\rm T$}\hbox{\hbox
to0pt{\kern0.3\wd0\vrule height0.9\ht0\hss}\box0}}}}
\def\bbbs{{\mathchoice
{\setbox0=\hbox{$\displaystyle     \rm S$}\hbox{\raise0.5\ht0\hbox
to0pt{\kern0.35\wd0\vrule height0.45\ht0\hss}\hbox
to0pt{\kern0.55\wd0\vrule height0.5\ht0\hss}\box0}}
{\setbox0=\hbox{$\textstyle        \rm S$}\hbox{\raise0.5\ht0\hbox
to0pt{\kern0.35\wd0\vrule height0.45\ht0\hss}\hbox
to0pt{\kern0.55\wd0\vrule height0.5\ht0\hss}\box0}}
{\setbox0=\hbox{$\scriptstyle      \rm S$}\hbox{\raise0.5\ht0\hbox
to0pt{\kern0.35\wd0\vrule height0.45\ht0\hss}\raise0.05\ht0\hbox
to0pt{\kern0.5\wd0\vrule height0.45\ht0\hss}\box0}}
{\setbox0=\hbox{$\scriptscriptstyle\rm S$}\hbox{\raise0.5\ht0\hbox
to0pt{\kern0.4\wd0\vrule height0.45\ht0\hss}\raise0.05\ht0\hbox
to0pt{\kern0.55\wd0\vrule height0.45\ht0\hss}\box0}}}}
\def\bbbz{{\mathchoice {\hbox{$\sf\textstyle Z\kern-0.4em Z$}}
{\hbox{$\sf\textstyle Z\kern-0.4em Z$}}
{\hbox{$\sf\scriptstyle Z\kern-0.3em Z$}}
{\hbox{$\sf\scriptscriptstyle Z\kern-0.2em Z$}}}}

 \newtheorem{thm}{Theorem}
 \newtheorem{cor}[thm]{Corollary}
 \newtheorem{lem}[thm]{Lemma}
 
 \theoremstyle{definition}
 
 \theoremstyle{remark}

\def\cX{{\mathcal X}}

\def\({\left(}
\def\){\right)}
\def\[{\left[}
\def\]{\right]}
\def\<{\langle}
\def\>{\rangle}

\def\fB{{\mathfrak B}}

\def\F{\mathbb{F}}

\def\Z{\mathbb{Z}}

\def\R{\mathbb{R}}

\def\ep{{\mathbf{\,e}}_p}

\begin{document}

\title[Solutions to Polynomial 
Congruences in Small Boxes]{On Solutions to Some Polynomial 
Congruences in Small Boxes}

\author{Igor E.~Shparlinski}
\address{Department of Computing, Macquarie University, Sydney, NSW 2109, Australia}

\email{igor.shparlinski@mq.edu.au}

\begin{abstract}  
We use bounds of mixed  character sum 
to study the distribution of solutions to certain polynomial systems of 
congruences modulo a prime $p$. In particular, we obtain nontrivial 
results about the number of solution in boxes with the 
side length below $p^{1/2}$, which seems to be the limit of 
more general  methods
based on the bounds of exponential sums
along varieties.  
\end{abstract}



\maketitle

\section{Introduction}

 There is an extensive literature investigating the distribution of solutions
 to the system congruence
\begin{equation}
\label{eq:syst}
F_j(x_1, \ldots, x_n) \equiv 0 \pmod p, \qquad j =1, \ldots, m,
\end{equation}
$F_j(X_1,\ldots,X_n) \in \Z[X_1,\ldots,X_n]$, $j=1, \ldots, m$,
in $m$ variables
with integer coefficients, 
modulo a prime $p$, see~\cite{Fouv,FoKa,Luo,ShpSk,Skor}.

In particular, subject some additional condition
 (related to the so-called $A$-number),  
Fouvry and Katz~\cite[Corollary~1.5]{FoKa} have given 
an asymptotic formula for the number of solutions to~\eqref{eq:syst}
in a box 
$$(x_1, \ldots, x_n) \in [0, h-1]^n
$$ 
for a rather small $h$. In fact the 
limit of the method of~\cite{FoKa} is $h = p^{1/2 + o(1)}$. 

Here we consider a
very special class of systems of $s+1$ polynomial congruences
\begin{equation}
\label{eq:prod}
x_1 \ldots x_n  \equiv a \pmod p, 
\end{equation}
and 
\begin{equation}
\label{eq:diag}
c_{1,j}x_1^{k_{1,j}} + \ldots + c_{n,j} x_n^{k_{m,j}} \equiv b_j \pmod p, 
\qquad j =1, \ldots, s,
\end{equation}
where $a,b_j, c_{i,j}, k_{i,j} \in \Z$,
with $\gcd(a c_{i,j},p) = 1$,  $i=1, \ldots, n$, $j =1, \ldots, s$, and $3 \le k_{i,1} < \ldots < k_{i,s}$.

The interest to the systems of congruences~\eqref{eq:prod} and~\eqref{eq:diag}
stems from the work of Fouvry and Katz~\cite{FoKa}, where a particular case
of the congruence~\eqref{eq:prod} and just one congruence
of the type~\eqref{eq:diag} (that is, for $s=1$) with the same odd exponents 
$k_{1,1} = \ldots = k_{n,1} = k$ and $b_1 = 0$ is given as an example of 
a variety to which one of their main general results applies. 
In particular, in this case and for $k \ge 3$, $b_1=0$ 
(and fixed non-zero coefficients) we see that~\cite[Theorem~1.5]{FoKa}
gives an asymptotic for the number of solutions
with $1 \le x_i \le h$, $i=1, \ldots, n$,  starting from the 
values of $h$   of size about  
$\max\{p^{1/2 + 1/n}, p^{3/4}\} \log p$.
Here we show that a different and more specialised treatment allows
to significantly lower this threshold, which now in some cases reaches 
 $p^{1/4+\kappa}$ for any $\kappa > 0$. 
Furthermore, this applies to the systems~\eqref{eq:prod} and~\eqref{eq:diag}
in full generality and is uniform with respect to the coefficients. 

More precisely, we use a combination of
\begin{itemize}
\item  the bound of mixed character sums to due to Chang~\cite{Chang};
\item  the result of Ayyad,  Cochrane, and  Zheng~\cite{ACZ} on the fourth 
moment o short character sums;
\item the bound of Wooley~\cite{Wool3}  on exponential sums with polynomials.
\end{itemize}

We note that the classical P{\'o}lya-Vinogradov 
and Burgess bounds of multiplicative character sums
(see~\cite[Theorems~12.5 and 12.6]{IwKow}) in a combination with 
a result of Ayyad,  Cochrane, and  Zheng~\cite{ACZ}, 
has been used in~\cite{Shp1,Shp2}
to study the distribution of the single congruence~\eqref{eq:prod} in 
very small boxes), and thus go below the $p^{1/2}$-threshold. 

Here we show that the recent result of Chang~\cite{Chang}
enables us now to study a much more general case of the
simultaneous congruences~\eqref{eq:prod} and~\eqref{eq:diag}.


Throughout the paper, the implied constants in the symbols ``$O$'' and
``$\ll$'' can depend on the degrees $k_{i,j}$ in~\eqref{eq:prod} 
and~\eqref{eq:diag} as well as, occasionally, of 
some other polynomials involved.
We recall that
the expressions $A \ll B$ and $A=O(B)$ are each equivalent to the
statement that $|A|\le cB$ for some constant $c$.

\section{Character and Exponential Sums} 

Let $\cX_p$ be the set of multiplicative characters modulo $p$
and let $\cX_p^* = \cX_p \setminus \{\chi_0\}$ be the 
set of non-principal characters. 
We also denote
$$
\ep(z) = \exp(2 \pi i z/p).
$$
We appeal to~\cite{IwKow} for a background 
on the basic properties of multiplicative characters and
exponential functions, such as orthogonality. 

The following bounds of exponential sums twisted with a
multiplicative character has been given by Chang~\cite{Chang}
for sum in arbitrary finite  fields but only for intervals starting 
at the origin. However, a simple examination of the argument of~\cite{Chang}
reveals that this is not important for the proof:

\begin{lem}
   \label{lem:Chang}
For  any character $\chi \in \cX_p^*$,  a polynomial $F(X) \in \Z[X]$
of degree $k$ 
and any integers $u$ and $h\ge p^{1/4+\kappa}$, we have
$$\sum_{x=u+1}^{u+h} \chi(x) \ep(F(x)) \ll 
h p^{-\eta}, 
$$
where
$$
\eta = \frac{\kappa^2}{4(1+2 \kappa)(k^2+2k + 3)}.
$$
\end{lem}

We note that we do not impose any conditions on the polynomial 
$F$ in Lemma~\ref{lem:Chang}.

On the other hand when $\chi = \chi_0$, we use the following
a very special case of the much more general bound
of Wooley~\cite{Wool3} that applies to polynomials with arbitrary real coefficients.

\begin{lem}
\label{lem:Wooley}
For  any polynomial $F(X) \in \Z[X]$
of degree $k > 2$ with the leading coefficient $a_k \not \equiv 0 \pmod p$, 
and any integers $u$ and $h$ with
$ h < p$, we have
$$\sum_{x=u+1}^{u+h} \ep(F(x)) \ll 
h^{1-1/ 2k(k-2)} + h^{1-1/2(k-2)}p^{1/ 2k(k-2)}.
$$
\end{lem}

Clearly Lemma~\ref{lem:Wooley} is nontrivial only for $h \ge p^{1/k}$ which is 
actually the best possible range.  Furthermore, in a slightly shorter range
we have:

 \begin{cor}
\label{cor:Wooley short}
For  any polynomial $F(X) \in \Z[X]$
of degree $k > 2$ with the leading coefficient $a_k \not \equiv 0 \pmod p$, 
and any integers $u$ and $h$ with
$p^{1/(k-1)}\le  h < p$, we have
$$\sum_{x=u+1}^{u+h} \ep(F(x)) \ll h^{1-1/ 2k(k-2)}.
$$
\end{cor}

We make use of  the following estimate  
of Ayyad, Cochrane and Zheng~\cite[Theorem~1]{ACZ}.

 \begin{lem}
\label{lem:ACZ} Uniformly over  integers $u$ and $h\le p$,  the congruence
$$
 x_1 x_2  \equiv  x_3 x_4   \pmod p, \qquad   u+1 \le x_1,x_2,x_3,x_4 \le u+h,
$$
has   $h^{4}/p + O(h^{2}p^{o(1)})$ solutions as $h\to\infty$.
\end{lem}

We note that Lemma~\ref{lem:ACZ}  is a essentially a statement 
about the fourth monent of short character sums, see~\cite[Equation~(4)]{ACZ}.
In fact, the next result makes it clearer:

 \begin{cor}
\label{cor:ACZ-weights} Let $\rho(x)$ be an arbitrary complex valued
function with 
$$
|\rho(x)| \le 1, \qquad  x\in \R.
$$ 
Uniformly over  integers $u$ and $h\le p$,  we have
$$
\sum_{\chi\in \cX_p} \left|
\sum_{x=u+1}^{u+h} \rho(x) \chi(x)\right|^4 \le h^{4} + O\(h^{2}p^{1+o(1)}\), 
$$
as $h\to\infty$.
\end{cor}

\begin{proof}
Expanding the fourth power, and changing the order of summation, we obtain 
\begin{equation*}
\begin{split}
\sum_{\chi\in \cX_p} &\left|
\sum_{x=u_i+1}^{u+h} \rho(x) \chi(x)\right|^4 \\
&= \sum_{\chi\in \cX_p} \sum_{x_1, \ldots, x_4=u+1}^{u+h} 
\rho(x_1) \rho(x_2) \rho(x_3) \rho(x_4) \chi(x_1x_2 x_3^{-1} x_4^{-1})\\
&= \sum_{x_1, \ldots, x_4=u+1}^{u+h} 
\rho(x_1) \rho(x_2) \rho(x_3) \rho(x_4) 
\sum_{\chi\in \cX_p} \chi(x_1x_2 x_3^{-1} x_4^{-1}).
 \end{split}
\end{equation*}
Using the orthogonality of characters, we write
\begin{equation*}
\begin{split}
\sum_{\chi\in \cX_p} &\left|
\sum_{x=u+1}^{u+h} \rho(x) \chi(x)\right|^4 \\
&= \sum_{\chi\in \cX_p} \sum_{x_1, \ldots, x_4=u+1}^{u+h} 
\rho(x_1) \rho(x_2) \rho(x_3) \rho(x_4) \chi(x_1x_2 x_3^{-1} x_4^{-1})\\
&= (p-1) \sum_{\substack{x_1, \ldots, x_4=u+1\\ x_1x_2 \equiv x_3 x_4 \pmod p}}^{u+h} 
\rho(x_1) \rho(x_2) \rho(x_3) \rho(x_4) \\
& \le (p-1)  \sum_{\substack{x_1, \ldots, x_4=u+1\\ x_1x_2 \equiv x_3 x_4 \pmod p}}^{u+h} 1.
 \end{split}
\end{equation*}
Using Lemma~\ref{lem:ACZ} we derive the desired bound. 
\end{proof}

\section{Main Result}

We are now able to present our main result.
Let $\fB$ be a cube of the form
$$
\fB = [u_1+1,u_1+h]\times \ldots \times [u_n+1,u_n+h]
$$
with some integers $h,u_i$ with 
$1 \le u_i +1 < u_i+h < p$, $i=1, \ldots, n$. 
We denote by $N(\fB)$ the number of integer vectors
$$
(x_1, \ldots, x_n) \in \fB
$$
satisfying~\eqref{eq:prod} and~\eqref{eq:diag} simultaneously. 

As we have mentioned the case of just one congruence~\eqref{eq:prod}
has been considered in~\cite{Shp1,Shp2}, so we always assume that $s \ge 1$
(and thus $n \ge 3$). 

Let
\begin{equation*}
\begin{split}
k & = \min\{k_{i,j}~:~i =1, \ldots, n, \ j =1, \ldots, s\},\\
K & = \max\{k_{i,j}~:~i =1, \ldots, n, \ j =1, \ldots, s\}.
\end{split}
\end{equation*}

\begin{thm}
\label{thm:N Asymp} For any fixed $\kappa > 0$ and 
$$
p > h \ge \min\{ p^{1/4+\kappa}, p^{1/(k-1)}\}
$$
we have  
$$ 
N_p(\fB) =  \frac{h^n}{p^{s+1}} + 
O\( h^{n} p^{-1-\eta(n-4)} +  h^{n-2} p^{-\eta(n-4)}\), 
$$
where
$$
\eta = \frac{\kappa^2}{4(1+2 \kappa)(K^2+2K + 3)}.
$$
\end{thm}

\begin{proof} Using the orthogonality of characters, 
we write
\begin{equation*}
\begin{split}
N_p(\fB) = \sum_{(x_1, \ldots, x_n) \in \fB}
 \frac{1}{p^s}
\sum_{\lambda_1, \ldots, \lambda_s =0}^{p-1} &
\ep\(  \sum_{j=1}^s  \lambda_j \(\sum_{i=1}^n c_{i,j}x_i^{k_{i,j}}- b_j\)\) \\
& \qquad \qquad  \quad \frac{1}{p-1}  \sum_{\chi\in \cX_p} \chi(x_1 \ldots x_na^{-1}) .
\end{split}
\end{equation*}
Hence, changing the order of summation, we obtain 
\begin{equation*}
\begin{split}
N_p(\fB) =
 \frac{1}{(p-1) p^s} & \sum_{\lambda_1, \ldots, \lambda_s =0}^{p-1} 
  \ep\( - \sum_{j=1}^s  \lambda_j  b_j \)\\
 &\sum_{\chi\in \cX_p} \chi(a^{-1}) 
\prod_{i=1}^n S_i(\chi; \lambda_1, \ldots, \lambda_s),
\end{split}
\end{equation*}
where
$$
S_i(\chi; \lambda_1, \ldots, \lambda_s) = \sum_{x=u_i+1}^{u_i+h} 
\ep\(\sum_{j=1}^s \lambda_j c_{i,j} x^{k_{i,j}}\), \quad i =1, \ldots, n.
$$
Separating the term $h^n/(p-1) p^s$, 
corresponding to $\chi=\chi_0$ and $\lambda_1 =\ldots= \lambda_s =0$,
we derive
\begin{equation}
\label{eq:R1R2}
N_p(\fB) - \frac{h^n}{(p-1) p^s} \ll  \frac{1}{p^{s+1}}\(R_1  + R_2\), 
\end{equation}
where
\begin{equation*}
\begin{split}
R_1 &=   \sum_{\lambda_1, \ldots, \lambda_s =0}^{p-1} \sum_{\chi\in \cX_p^*}
 \prod_{i=1}^n |S_i(\chi; \lambda_1, \ldots, \lambda_s)|,\\
 R_2 &=   \sum_{\substack{\lambda_1, \ldots, \lambda_s =0\\
 (\lambda_1, \ldots, \lambda_s) \ne (0, \ldots, 0)}}^{p-1} 
 \prod_{i=1}^n |S_i(\chi_0; \lambda_1, \ldots, \lambda_s)|.
\end{split}
\end{equation*}

To estimate $R_1$ we use Lemma~\ref{lem:Chang} and write
$$
R_1 \le  h^{n-4} p^{-\eta(n-4)}  \sum_{\lambda_1, \ldots, \lambda_s =0}^{p-1} \sum_{\chi\in \cX_p^*}
 \prod_{i=1}^4 |S_i(\chi; \lambda_1, \ldots, \lambda_s)|.
$$
Using the H{\"o}lder inequality and Corollary~\ref{cor:ACZ-weights}, we obtain 
\begin{equation*}
\begin{split}
 \sum_{\chi\in \cX_p^*}
 \prod_{i=1}^4 |S_i(\chi; \lambda_1, \ldots, \lambda_s)| 
 \le  \(\prod_{i=1}^4  \sum_{\chi\in \cX_p^*}  
 |S_i(\chi; \lambda_1, \ldots, \lambda_s)|^4\)^{1/4}&\\
 \ll  h^{4}  + h^{2}&p^{1+o(1)} .
\end{split}
\end{equation*}
Therefore, 
\begin{equation}
\label{eq:R1 bound}
R_1 \ll  h^{n} p^{s-\eta(n-4)} + h^{n-2} p^{s+1-\eta(n-4)}.  
\end{equation}

Furthermore, for $R_2$ we use Corollary~\ref{cor:Wooley short} to derive
$$
R_2 \le  h^{(n-2)(1-1/2K(K-2))} \sum_{\substack{\lambda_1, \ldots, \lambda_s =0\\
 (\lambda_1, \ldots, \lambda_s) =(0, \ldots, 0)}}^{p-1} 
 \prod_{i=1}^2 |S_i(\chi; \lambda_1, \ldots, \lambda_s)|.
$$
Using the H{\"o}lder inequality and the orthogonality of exponential functions 
(similarly to the proof of Corollary~\ref{cor:ACZ-weights}), we obtain 
\begin{equation*}
\begin{split}
 \sum_{\substack{\lambda_1, \ldots, \lambda_s =0\\
 (\lambda_1, \ldots, \lambda_s) =(0, \ldots, 0)}}^{p-1} &
 \prod_{i=1}^2 |S_i(\chi; \lambda_1, \ldots, \lambda_s)|\\
 &\le  \(\prod_{i=1}^2   \sum_{\lambda_1, \ldots, \lambda_s =0}^{p-1} 
 |S_i(\chi; \lambda_1, \ldots, \lambda_s)|^2\)^{1/2} \ll  p^sh.
\end{split}
\end{equation*}
Thus 
\begin{equation}
\label{eq:R2 bound}
R_2 \ll    h^{n -1 - (n-2)/ 2K(K-2)} p^{s}.
\end{equation}
Substituting the bounds~\eqref{eq:R1 bound} and~\eqref{eq:R2 bound} in~\eqref{eq:R1R2}
we obtain
\begin{equation*}
\begin{split}
N_p(\fB) - & \frac{h^n}{p^{s+1}} \\
&\ll 
h^{n}  p^{-1-\eta(n-4)}  +  h^{n-2} p^{-\eta(n-4)}+ h^{n -1 - (n-2)/ 2K(K-2)} p^{-1} .\end{split}
\end{equation*}
Clearly, 
$$
\eta <  \frac{1}{2K(K-2)}.
$$
Thus we see that the second term always dominates  
the third term and the result follows. 
\end{proof}

\section{Comments}

Clearly, for any $\kappa > 0$, $k\ge 5$ and $p > h \ge p^{1/4+\kappa}$,
Theorem~\ref{thm:N Asymp} implies that 
$$
N_p(\fB)= (1 + o(1))  \frac{h^n}{p^{s+1}},
$$
as $p\to \infty$, provided that 
$$
n \ge (s+1/2) \eta^{-1}+4.
$$
For $k=3$ and $4$ the range of Theorem~\ref{thm:N Asymp} 
becomes $h \ge p^{1/2}$ and $h \ge p^{1/3}$. However it is easy to 
see that using the full power of Lemma~\ref{lem:Wooley} instead of 
Corollary~\ref{cor:Wooley short} one can derive nontrivial results in 
a wider range.
Namely, for any $\kappa > 0$ there exists some 
$\gamma > 0$ (independent on $n$ and other parameters in~\eqref{eq:prod} and~\eqref{eq:diag})  
such that, for $h \ge p^{1/3+\kappa}$ if $k = 3$ and 
for $h \ge p^{1/4+\kappa}$ if $k = 4$, we have
$$
N_p(\fB)=  \frac{h^n}{(p-1)p^{s}} + O\(h^{(1 - \gamma)n}\).
$$
We also recall that for polynomials of small degrees
 stronger values of Lemma~\ref{lem:Wooley}
are available, see~\cite{BoWo} and references therein.

Note that the same method can be applied (with essentially the same results) 
to the systems of congruences where instead of~\eqref{eq:prod} 
we have a more general congruence 
$$
x_1^{m_1} \ldots x_n^{m_n}  \equiv a \pmod p
$$
for some integers $m_i$ with $\gcd(m_i,p-1)=1$, $i=1, \ldots, n$.

Moreover, we recall that the  Weil bound~\cite[Appendix~5, Example~12]{Weil} (see
also~\cite[Chapter~6, Theorem~3]{Li})  and the standard reduction between 
complete and incomplete sums (see~\cite[Section~12.2]{IwKow})
implies that 
$$\sum_{x=u+1}^{u+h} \chi(G(x)) \ep(F(x)) \ll p^{1/2} \log p, 
$$
where $G(x)$ is a polynomial that is not a perfect power of any other polynomial
in the algebraic closure $\overline \F_p$ of the finite field of $p$ elements,
Thus for $h \ge p^{1/2+\kappa}$, using this bound instead of Lemma~\ref{lem:Chang}
allows us to replace~\eqref{eq:prod} with the congruence 
$$
G_1(x_1) \ldots G_n(x_n) \equiv a \pmod p
$$
for arbitrary polynomials $G_1(X), \ldots, G_n(X)\in \Z[X]$ such that their 
reductions modulo $p$ are not perfect powers in $\overline \F_p$.
In fact, even for $G_1(X) = \ldots= G_n(X)= X$ (that is, for the 
congruence~\eqref{eq:prod}) this leads to a result, which is sometimes stronger 
that those of~\cite{FoKa} and Theorem~\ref{thm:N Asymp}.

\section{Acknowledgment}

The author is very grateful to Mei-Chu Chang for the confirmation 
that the main result of~\cite{Chang}  applies to intervals in 
an arbitrary position.  

This work was supported in part by the  ARC Grant DP1092835.


\begin{thebibliography}{1}

\bibitem{ACZ} A. Ayyad, T. Cochrane, and Z. Zheng, 
`The congruence $x_1x_2 \equiv x_3x_4 \pmod p$, the equation
$x_1x_2 = x_3x_4$ and the mean value of character sums', 
{\it J. Number Theory\/}, \textbf{59} (1996), 398--413.


\bibitem{BoWo}  K. D. Boklan, and  T. D. Wooley,
`On Weyl sums for smaller exponents', {\it Funct. et Approx.
Commen.  Math.\/}, {\bf 46} (2012), 91--107. 


\bibitem{Chang} M.-C. Chang, `An estimate of incomplete 
mixed character
sums', {\it 
An Irregular Mind\/}, Bolyai Society
Math. Studies, vol.~21, Springer, Berlin, 2010, 243--250.

\bibitem{Fouv} {\'E}. Fouvry, `Consequences of a result of N. Katz
and G. Laumon concerning trigonometric sums', {\it Israel J.
Math.\/}, {\bf 120} (2000), 81--96.

\bibitem{FoKa} {\'E}. Fouvry and  N. Katz, `A general
stratification theorem for exponential sums, and applications',
{\it J. Reine Angew. Math.\/}, {\bf 540} (2001), 115--166.

\bibitem{IwKow} H. Iwaniec and E. Kowalski,
{\it Analytic number theory\/}, Amer.  Math.  Soc.,
Providence, RI, 2004.


\bibitem{Li}  W.-C. W. Li, {\it Number theory with applications\/},
World Scientific,  Singapore, 1996.


\bibitem{Luo} W. Luo, `Rational points on complete
intersections over $\F_p$',
{\it Internat. Math. Res. Notices\/},
{\bf 1999} (1999),   901--907.

\bibitem{Shp1} I. E. Shparlinski, `On the distribution of points on
multidimensional modular hyperbolas',
{\it Proc. Japan Acad. Sci., Ser.A\/}, {\bf 83} (2007), 5--9.

\bibitem{Shp2} I. E. Shparlinski, `On a generalisation of a Lehmer
problem', {\it Math. Zeitschrift\/},  
{\bf 263} (2009), 619--631.

\bibitem{ShpSk} I. E. Shparlinski and A. N. Skorobogatov,
`Exponential sums and rational points on complete
intersections', {\it Mathematika\/}, {\bf 37} (1990), 201--208.

\bibitem{Skor}   A. N. Skorobogatov,
`Exponential sums, the geometry of hyperplane sections,
and some Diophantine problems', {\it Israel J. Math.\/},
{\bf  80} (1992), 359--379.

\bibitem{Weil} A. Weil,
{\it Basic number theory\/}, Springer-Verlag, New York, 1974.
%


\bibitem{Wool3} T.~D.~Wooley,
`Vinogradov's mean value theorem via efficient congruencing, II',
{\it Preprint\/} 2011, (available from {\tt http://arxiv.org/abs/1112.0358}).


\end{thebibliography}
\end{document}